\documentclass[a4paper,11pt]{amsart}
\usepackage{euscript,
mathrsfs,color,latexsym,marginnote}
\usepackage{graphicx}
\usepackage{amsmath,amssymb,amsthm,amsfonts}
\usepackage{amssymb}
\usepackage[active]{srcltx}
\usepackage{geometry}
\usepackage{hyperref}
\usepackage{color}

\newtheorem{thm}{Theorem}[section]

\newtheorem{lem}[thm]{Lemma}
\newtheorem{prop}[thm]{Proposition}
\newtheorem{rem}[thm]{Remark}

\newcommand{\R}{\mathbb R}

\begin{document}
	
	\title[Monotonicity and Liouville-type theorems]{Monotonicity and  Liouville-type theorems for semilinear elliptic problems in the half space}
	
	\author{Berardino Sciunzi and Domenico Vuono $^{*}$}

	\email[Berardino Sciunzi]{berardino.sciunzi@unical.it}
    \email[Domenico Vuono]{domenico.vuono@unical.it}
	\address{$^{*}$Dipartimento di Matematica e Informatica, Università della Calabria,
		Ponte Pietro Bucci 31B, 87036 Arcavacata di Rende, Cosenza, Italy}
	

	

	\begin{abstract}
		We consider classical solutions to $-\Delta u = f(u)$  in half-spaces, under homogeneous Dirichlet boundary conditions.  We prove that any positive solution is strictly monotone increasing in the direction orthogonal to the boundary, provided that it is directionally bounded on finite strips. As a corollary, we deduce a new Liouville-type theorem for the Lane-Emden equation.
	\end{abstract}

 	\maketitle
  
	\section{Introduction}
	Our paper is mainly concerned with proving monotonicity properties of positive classical  solutions to $ -\Delta u = f(u)$ under zero Dirichlet assumptions in half spaces. In this context we improve earlier results by weakening the boundedness assumption on the solution. As a corollary, we deduce new Liouville-type theorems for the Lane-Emden equation. We prefer to start the presentation of the paper with the application to the Lane-Emden problem:
\begin{equation}\label{eq:laneemden}
\begin{cases}
        -\Delta u = u^{q} & \mbox{in $\R^n_+$,}\\
        u = 0 & \mbox{on $\partial\R^n_+$,},
    \end{cases}
\end{equation}
where $n\ge 2$ and $q>1$. We assume, with no loss of generality, that $\R^n_+=\{x_n>0\}$. The study of \eqref{eq:laneemden} has a long history. It is conjectured that the only nonnegative solution to \eqref{eq:laneemden} is the trivial one, namely $u \equiv 0$.

In their seminal work \cite{GS}, Gidas and Spruck proved that this is indeed the case provided $1<q\leq q_s(n)$, where $q_s(n):=(n+2)/{(n-2)_+}$.

For supercritical exponents $q > q_S(n)$, only partial results are known. Later on, Dancer \cite{Dancer} proved that bounded solutions are monotone in the normal direction and deduced that no nontrivial bounded solution exists in the range $1<q<q_D(n):=({n+1})/({(n-3)_+})$.  We emphasize that monotonicity properties of positive solutions have been well understood in the celebrated  works \cite{BCN1,BCN2,BCN3}.
Later on, Farina \cite{Farina1} sharpened this result by showing that the only bounded nonnegative solution is $u \equiv 0$ provided
$1<q < q_{JL}(n-1),$
and more generally, the same holds for solutions which are stable outside a compact set. Here $q_{JL}$ denotes the Joseph--Lundgren stability exponent, defined by
$q_{JL}(n) := ((n-2)^2 - 4n + 8\sqrt{\,n-1})/({(n-2)(n-10)_+}).$ For results concerning Liouville-type theorems and related developments, we refer the reader to \cite{DF,Farina2}.

More recently, Chen, Lin and Zou \cite{CLZ} proved that no bounded nonnegative solution $u \not\equiv 0$ of \eqref{eq:laneemden} exists for any $q>1$.  

Recently, Dupaigne, Sirakov and Souplet \cite{DSS} showed that, more generally, no nontrivial monotone solution of \eqref{eq:laneemden} exists, whether bounded or not. The result was extended to the more general case of stable solutions in \cite{DFT}. The results in \cite{DFT,DSS} are, till now, the more general ones in the literature. In spite of this, the problem of the existence of the solutions is still open in its full generality.
\begin{rem}\label{remstrips}
Let us stress the fact that,  for any $q>1$, problem \eqref{eq:laneemden} does not admit any positive classical solution which is bounded on finite strips as a corollary. In this case, in fact, any solution is \textit{monotone} and, therefore, \textit{stable}.  The case $n=2$ does not require this assumption, see \cite{DS,FS}.
\end{rem}

\noindent The purpose of our work is to improve upon previous results and, in particular, to refine the boundedness assumption on finite strips.  Throughout the paper, we shall only assume that the solution is \textit{directionally bounded} on finite strips, namely we shall assume up to rotations:
\begin{equation}\label{assunzionesullau}\tag{$\mathcal{H}_u$}
   u \in L^{\infty}(\mathbb{R}^{\,n-2} \times K), 
    \quad \text{for every compact set } K \subseteq \{\,x'=0,\; x_n \geq 0\,\},
\end{equation}

\

\noindent where here the $n-2$ variables are precisely $x'=(x_2,\dots,x_{n-1})$.\\

\noindent Our main application is stated in the following:

    \begin{thm}\label{teo:Lane-Emden}
        If $u$ is a nonnegative solution to \eqref{eq:laneemden} which is \textit{directionally bounded} on finite stripes (it fulfils \eqref{assunzionesullau} up to rotations), then $u\equiv 0$.
    \end{thm}

\begin{rem}\label{ghgjkgkhg}
	It is clear that any solution which is bounded on finite strips, automatically fulfils \eqref{assunzionesullau} up to rotations. Thus, although we could not answer completely to the conjecture, our result improves the ones in \cite{DFT,DSS} (see remark \ref{remstrips}), actually also thanks to the deep results therein. 
\end{rem}

Theorem \ref{teo:Lane-Emden} follows from a very more general result regarding the monotonicity of the solutions in half-spaces, for a general class of semilinear problems. Our main result, in fact, states that under the assumption \eqref{assunzionesullau}, any solution $u$ is monotone with respect to the $x_n$-variable for the class of problems:

    \begin{equation}\label{eq:principale}
\begin{cases}
        -\Delta u = f(u) & \mbox{in $\R^n_+$,}\\
        u(x_1,x',x_n)\geq 0 & \mbox{in $\R^n_+$,}\\
        u(x_1,x',x_n)=0 & \mbox{on $\partial \R^n_+$,}
    \end{cases}
    \end{equation}
where $n\geq 2$ and $f(\cdot)$ satisfies: 
\begin{itemize}
    \item [($h_f$)] the function $f: \R^+\cup \{0\}\rightarrow \R$ is locally Lipschitz continuous and $$\lim_{t\rightarrow 0^+} \frac {f(t)}{t}=f_0\in \R^+\cup \{0\}.$$
\end{itemize}
In the following, we denote a generic point in $\R^n$ by $(x_1,x',x_n)$ with $x'=(x_2,...,x_{n-1})\in\R^{n-2}$. We have the following result.

\begin{thm}\label{teo:Monotonia}
Let $u$ be a positive solution of \eqref{eq:principale}, where $f$ satisfies the assumption $(h_f)$. 
If $u$ is directionally bounded on finite strips (it fulfils \eqref{assunzionesullau} up to rotations), then $u$ is monotone increasing in the $x_N$-direction with
\[
\frac{\partial u}{\partial x_N} > 0 
\qquad \text{in } \mathbb{R}^N_+.
\]
\end{thm}

The study of the monotonicity of the solutions was started in the semilinear nondegenerate case in a series of papers. We refer to \cite{BCN1,BCN2,BCN3}
and to \cite{Dancer,Dancer2}. In view of Remark \ref{ghgjkgkhg}, it is easy to see that our result improves all the earlier ones. 

 For previous results concerning the monotonicity of solutions in half-spaces, either in the non-degenerate case or when the nonlinearity includes a singular component, we refer the reader to \cite{FMS1,FMS3,FMS2,FMRS,MMS1,MMS2,MMS3}.\\

\noindent The main tool we employ here to obtain our results is the classical moving plane method, originally introduced in \cite{Alex,Serrin}. In particular, to achieve our goals, we use a rotating plane technique that goes back to  \cite{BCN2,DS} and has been refined in  \cite{FS}. All the approaches and the results in \cite{BCN2,DS,FS} are actually restricted to the two dimensional case. There is a very strong advantage when working in the plane since, in this case, it is possible to reduce to work in bounded domains thanks to the geometric nature of the rotating technique. In higher dimension this is no more possible and we shall make some effort that can be more appreciated while reading the paper. Our method is described and developed in Section \ref{SezionePreliminare}.\\

\textbf{Strategy of the proof.}
The proof of our main results is based on a refined version of the moving plane method. 
In particular, we exploit a rotating plane technique, following a similar approach as in 
\cite{DS,FS}. Our aim is to show that 
\[
u \leq u_{\lambda} \quad \text{in } \Sigma_{\lambda}, 
\qquad \Sigma_{\lambda}:=\{0\le x_n \le \lambda\},
\]
for every $\lambda>0$, where $u_\lambda$ denotes the reflection of $u$ with respect to the hyperplane $\{x_n=\lambda\}$.  

\smallskip
\noindent
To achieve this, we first introduce a vector $V_\theta$ lying in the $(x_1,x_n)$-plane 
such that $\langle V_\theta,e_n\rangle >0$, and we denote by $\theta$ the angle between 
$V_\theta$ and $e_n$. For $h>0$, we consider the domain 
\[
\mathcal T_{\theta,h} = \R^{n-2} \times \hat T_{\theta,h},
\]
where $\hat T_{\theta,h}$ is the right triangle in the $(x_1,x_n)$-plane  with vertices $(0,0)$, $(h,0)$, and the third vertex chosen so that the hypotenuse is orthogonal to $V_\theta$.
(see Section~\ref{SezionePreliminare} for further details).  

\smallskip
\noindent
The first step is to prove the existence of $\bar\theta>0$ and $\bar h>0$, small enough, such that 
\[
u < u_{\bar\theta,\bar h} \quad \text{in } \mathcal T_{\bar\theta,\bar h},
\]
where $u_{\bar\theta,\bar h}$ denotes the reflection of $u$ with respect to the hyperplane 
$\{ \langle x-\bar h e_n, V_{\bar\theta}\rangle =0 \}$ (see Remark~\ref{monotonia_nel_triangolino}).

\smallskip
\noindent
By letting $\theta \to 0$, we deduce the monotonicity of $u$ in the $x_n$-direction near 
the boundary $\partial \R^n_+$ (see Proposition~\ref{monotonia_vicino_al_bordo}). This step crucially 
relies on Lemma~\ref{piccoleperubations} and Lemma~\ref{grandispostamenti}.  
Finally, to extend this monotonicity to the whole half-space $\R^n_+$ we argue by 
contradiction, where Lemma~\ref{lemmapreteoprincipale} plays a key role.  

\smallskip
\noindent

\section{Preliminary Results: The sliding-rotating technique}\label{SezionePreliminare}

We begin by introducing some notation and preliminary results. Throughout the paper, 
generic fixed or numerical constants will be denoted by $C$ (possibly with subscripts), 
and their values may vary from line to line or even within the same formula.
\vspace{0.2 cm}

For $0 \le \alpha < \beta$, we define the strip
\[
\Sigma_{(\alpha, \beta)} := \mathbb{R}^{N-1} \times (\alpha, \beta),
\]
and we denote
\[
\Sigma_\beta := \mathbb{R}^{N-1} \times (0, \beta)
\]
the strip corresponding to $\alpha = 0$.

Let $B^{''}(0,R)$ be the ball in $\mathbb{R}^{N-1}$ of radius $R$ centered at the origin. 
Then we define the cylinder
\begin{equation}\label{cilindro}
    \mathcal{C}_{(\alpha, \beta)}(R) = \mathcal{C}(R) := \Sigma_{(\alpha, \beta)} \cap \big( B^{''}(0,R) \times \mathbb{R} \big).
\end{equation}

\vspace{0.1cm}

\noindent For the proof of our results, the use of Harnack-type inequalities will play a crucial role. 
In particular, we will frequently rely on the classical Harnack inequality for Laplace 
equations (see \cite[Theorem 7.2.1]{PS} and the references therein). At a certain stage, 
as will become clear later, a boundary version of the Harnack inequality will be essential. 
For this reason, we state here a suitable adaptation of the more general and profound 
result by M.F.~Bidaut-Véron, R.~Borghol, and L.~Véron (see \cite[Theorem 2.8]{BBV}).

\begin{thm}[\cite{BBV}][Boundary Harnack Inequality]\label{Boundary_Harnack}
Let $R_0 > 0$ and define the cylinder $\mathcal{C}_{(0,L)}(2R_0)$. Let $u$ satisfy
\[
- \Delta u = c(x) u \quad \text{in } \mathcal{C}_{(0,L)}(2R_0),
\]
with $u$ vanishing on 
\[
\mathcal{C}_{(0,L)}(2R_0) \cap \{x_n = 0\},
\] 
and assume that 
\[
\|c(x)\|_{L^\infty(\mathcal{C}_{(0,L)}(2R_0))} \le C_0.
\]

Then there exists a constant $C = C(n, C_0)$ such that
\[
\frac{1}{C} \, \frac{u(z_2)}{\rho(z_2)} \le \frac{u(z_1)}{\rho(z_1)} \le C \, \frac{u(z_2)}{\rho(z_2)}, 
\quad \forall z_1, z_2 \in B_{R_0} \cap \mathcal{C}_{(0,L)}(2R_0) \text{ with } 0 < \frac{|z_2|}{2} \le |z_1| \le 2|z_2|,
\]
where $\rho(\cdot)$ denotes the distance function to $\partial \mathbb{R}^N_+$.
\end{thm}

We now state the following result, which is a principle in narrow domains 
and will play a crucial role in the forthcoming sections. For the proof, we refer to 
\cite[Theorem~1.1 and the subsequent remarks]{FMS1}.

\begin{prop}[\cite{FMS1}]\label{domini_piccoli}
   Assume that $n \geq 2$, and that $f$ is locally Lipschitz continuous. 
Let 
\[
\Sigma := \mathbb{R}^{n-k} \times \omega,
\]
where $\omega \subset \mathbb{R}^k$ is a measurable set. Consider $u,v \in C^{1,\alpha}_{\mathrm{loc}}(\Sigma)$ such that 
$u, \nabla u, v, \nabla v \in L^\infty(\Sigma)$ and
\[
\begin{cases}
-\Delta u \leq f(u) & \text{in } \Sigma, \\[0.3em]
-\Delta v \geq f(v) & \text{in } \Sigma, \\[0.3em]
u \leq v & \text{on } \partial \Sigma.
\end{cases}
\]

\noindent Then there exists $\delta_0 = \delta_0(n,\|\nabla u\|_\infty,\|\nabla v\|_\infty,
\|u\|_\infty,\|v\|_\infty,f) > 0$ such that, if the Lebesgue measure
\(\mathcal{L}(\omega) < \delta_0\), it follows that
\[
u \leq v \quad \text{in } \Sigma.
\]
\end{prop}

\subsection{The sliding-rotating technique}

 Let $\theta_1, \theta_n \in \mathbb{R}$ and set
\[
V_\theta := (\theta_1,0',\theta_n), \qquad 0'=(0,\dots,0)\in\mathbb{R}^{\,n-2}.
\]
The vector $V_\theta$ is chosen so that
$
\langle V_\theta, e_n\rangle > 0$ and $ 
\|V_\theta\|=1.$ We denote by $\theta$ the angle formed by $V_\theta$ and $e_n$, that is,
$$\cos \theta=\langle V_\theta,e_n\rangle=\theta_n.$$
For $h>0$ we consider the hyperplane orthogonal to $V_\theta$ and passing through the point $h e_n$, namely
\[
\mathcal{P}_{\theta,h} := \{\, x \in \mathbb{R}^n : \langle x - h e_n, V_\theta \rangle = 0 \,\}.
\]   
We denote by $\mathcal{T}_{\theta,h}\subset \R^2$ the open set delimited by $\mathcal{P}_{\theta,h}$, $\{x_1=0\}$ and $\{x_n=0\}$. Note that $\mathcal{T}_{\theta,h}$ can be written as 
\begin{equation}\label{ildominiotriangolo}
  \mathcal{T}_{\theta,h}=\mathbb{R}^{\,n-2}\times \hat T_{\theta,h}, \quad \text{with } x'=(x_2,...,x_{n-1})\in \R^{n-2},
\end{equation}

where $\hat T_{\theta,h}$ lies in the $(x_1,x_n)$-plane and is the right triangle bounded by the axes $\{x_1=0\},\{x_n=0\}$ and the line $\theta_1x_1+\theta_nx_n=\theta_nh$ (with vertices $(0,0)$, $(\tfrac{\theta_n}{\theta_1}h,0)$ and $(0,h)$ when $\theta_1\neq0$).  We also define $$u_{\theta,h}(x)=u\left(T_{\theta,h}(x)\right), \quad x\in \mathcal{T}_{\theta, h}$$
where $T_{\theta,h}(x)$ is the point symmetric to $x$ with respect to $\mathcal{P}_{\theta,h}$, and \begin{equation}\label{eq:2.2}
w_{\theta,h} := u - u_{\theta,h}.
\end{equation}

It is immediately clear that $u_{\theta,h}$ still fulfills $
-\Delta u_{\theta,h} = f\!\left(u_{\theta,h}\right),$ in the reflected domain,
and
\begin{equation}\label{eq:2.3}
-\Delta w_{\theta,h} = c_{\theta,h}\, w_{\theta,h}
\end{equation}
on the open set $\mathcal{T}_{\theta,h}$, where we have set
\begin{equation}\label{eq:2.4}
c_{\theta,h}(x) :=
\begin{cases}
\dfrac{f(u(x)) - f\!\left(u_{\theta,h}(x)\right)}{u(x)-u_{\theta,h}(x)} & 
\text{if } w_{\theta,h}(x) \neq 0, \\[2ex]
\quad 0 & \text{if } w_{\theta,h}(x) = 0.
\end{cases}
\end{equation}

Note that $|c_{\theta,h}| \leq C(\mathcal{T}_{\theta,h}, u, f)$ on the set $\mathcal{T}_{\theta,h}$,
where $C(\mathcal{T}_{\theta,h}, u, f)$ is a positive constant that can be determined by exploiting
the fact that $u$ and $u_{\theta,h}$ are bounded in the variables $x'=(x_2,...,x_{n-1})$ on $\mathcal{T}_{\theta,h}$ and $f$ is locally
Lipschitz continuous on $[0,+\infty)$. Let us remark that, exploiting the Strong Comparison Principle and the Dirichlet boundary condition, it follows that $w_{\theta,h}\leq 0$ on $\overline{\mathcal{T}_{h,\theta}}\cap \{x_n=0\}$, and $w_{\theta,h}$ is not identically zero on $\overline{\mathcal{T}_{h,\theta}}\cap \{x_n=0\}$.
\vspace{0.2cm}

We have the following:
\begin{lem}[Small perturbations]\label{piccoleperubations}
   Let $(\theta,h)$ and the set $\mathcal{T}_{\theta,h}$ be as above, and assume that
\begin{equation}\label{ipotesi_domini_piccoli}
    w_{\theta,h} < 0 \quad \text{in } \mathcal{T}_{\theta,h}, \quad \text{and} \quad
w_{\theta,h} \le 0 \quad \text{on } \partial \mathcal{T}_{\theta,h}.
\end{equation}
Then there exists a constant $\bar{\mu} = \bar{\mu}(\theta,h) > 0$ such that the following holds:  
if $(\theta',h')$ satisfies 
\[
|\theta - \theta'| + |h - h'| < \bar{\mu}
\quad \text{and} \quad 
w_{\theta',h'} \le 0 \ \text{on } \partial \mathcal{T}_{\theta',h'},
\]
then we also have
\[
w_{\theta',h'} < 0 \quad \text{in } \mathcal{T}_{\theta',h'}.
\]
\end{lem}

\begin{proof}
   We want to exploit Proposition \ref{domini_piccoli}. We consider the domain $\mathcal{T}_{\theta,h}$ given in \eqref{ildominiotriangolo}. Now pick a small $\epsilon=\epsilon (\theta,h)>0$ such that  $\mathcal{L}(\hat T_{\theta-\epsilon,h+\epsilon}\setminus\hat T_{\theta+\epsilon,h-\epsilon})<\delta_0/10$, and then a compact set $K\subset \hat T_{\theta+\epsilon,h-\epsilon}$ such that $\mathcal{L}(\hat T_{\theta+\epsilon,h-\epsilon}\setminus K)<\delta_0/10$, where $\delta_0$ is given by Proposition \ref{domini_piccoli}. Therefore, for all $(\theta',h')$ such that $|\theta-\theta'|+|h-h'|<\epsilon$, we have $\mathcal{L}(\hat T_{\theta',h'}\setminus K)<\delta_0/5$. Now we claim that there exist $\bar \mu \in (0,\epsilon)$, such that for all $(s',h')$ satisfying $|\theta-\theta'|+|h-h'|\le \bar\mu$, we have 
   \begin{equation}\label{compattini}
       w_{\theta',h'}<0 \quad \text{in } \R^{n-2}\times K.
   \end{equation}
   To prove the claim, we argue by contradiction. We assume that there exist $\mu_N\rightarrow 0$ and a sequence of points $x_N=(x_{1,N},x'_N,x_{n,N})\in \R^{n-2}\times K$ such that 
   \begin{equation}\label{contraddizione1}
       u(x_{1,N},x'_N,x_{n,N})\geq u_{\theta-\mu_N,h+\mu_N}(x_{1,N},x'_N,x_{n,N}).
   \end{equation} 
   Up to subsequences let assume that $(x_{1,N},x_{n,N})\rightarrow (\tilde x_1,\tilde x_n)\in K$.
We define the following sequence of  functions given by 
\begin{equation}\label{eq:rescaled1}
u_N(x_1,x',x_n) := \frac{u(x_1,x'+x'_N,x_n)}{u(0,x'_N,1)}.
\end{equation}
We note that $u_N(0,0',1) = 1$. Moreover, each $u_N$ satisfies
\begin{equation}\label{eq:un_eq}
-\Delta u_N(x) = c_N(x) \, u_N(x),
\end{equation}
where
\begin{equation}\label{eq:cn_def1}
c_N(x) := \frac{f(u(x_1,x'+x'_N,x_n))}{u(x_1,x'+x'_N,x_n)}.
\end{equation}

Since $u$ satisfies the assumption \eqref{assunzionesullau}, and by the fact that $f$ satisfies assumptions $(h_f)$, we have
\begin{equation}\label{eq:cn_bound1}
\|c_N\|_{L^\infty(\mathcal{K})} \le  C,
\end{equation}
for every compact set $\mathcal{K}\subset \overline{\R^n_+}$, where $C$ is a positive constant not depending on $N$.

 We consider $L$ large enough and we fix real numbers $ R, R_0$ such that
\begin{equation}\label{eq:3.71}
0 < 2R_0 < 1 < R < L.
\end{equation}
Our goal is to prove that
\[
\|u_N\|_{L^\infty(\mathcal{C}_{(0,L)}(R))} \leq C(L,R,R_0),
\]
where $\mathcal{C}_{(0,L)}(R)$ is defined as in \eqref{cilindro}.
Since $u_n(0,0',1)=1$, the classical Harnack inequality (see \cite[Theorem 7.2.1]{PS}) yields
\begin{equation}\label{eq:3.81}
\|u_N\|_{L^\infty(\mathcal{C}_{(0,L)}(R)\cap\{x_n \geq R_0/4\})} 
\leq C^{i}_H(L,R,R_0).
\end{equation}

In order to obtain a complementary bound in the region $\{x_n < R_0/4\}$, we apply Theorem~\ref{Boundary_Harnack}.  Let $\tilde{P}=(\tilde x_1,\tilde{x}',\tilde{x}_n)$ with $(\tilde x_1,\tilde{x}') \in B^{''}_R(0)$ and $0<\tilde{x}_n<R_0/4$.  
Choose a point
\[
\check{Q} = (\check{x}_1,\check{x}',0), \qquad (\check{x}_1,\check{x}') \in B^{''}_R(0),
\]
such that $\tilde{P} \in \partial B_{R_0}(\check{Q})$.  
Since $2R_0<R<L$, it is straightforward to verify that such a point exists.  

By \eqref{eq:cn_bound1}, we can apply Theorem \ref{Boundary_Harnack}, obtaining
\[
\frac{u_N(\tilde{P})}{\tilde{x}_n} \leq C \, \frac{u_N(\check{x}_1,\check{x}',R_0)}{R_0}.
\]
Moreover, since $u_N(x_1,x',0)=0$, this implies
\begin{equation}\label{eq:3.91}
    \|u_N\|_{L^\infty(\mathcal C_{(0,L)}(R)\cap\{x_n \leq R_0/4\})}
\leq C \cdot C^{i}_H(L,R,R_0).
\end{equation}
 
Combining \eqref{eq:3.81} and \eqref{eq:3.91}, we conclude that
\[
\|u_N\|_{L^\infty(\mathcal{C}_{(0,L)}(R))} \leq C(L,R,R_0).
\]

Next, we extend $u$ to the whole space $\mathbb{R}^N$ by odd reflection across $\{x_n=0\}$, 
which implies $
f(t) = -f(-t)$ for $t<0$.

In this setting we work with the cylinder
\[
\mathcal{C}_{(-L,L)}(R) := B^{''}_R(0) \times (-L,L).
\]

By standard regularity theory (see, e.g., \cite[Theorem 1]{GT}), the uniform $L^\infty$ bound implies
\[
\|u_N\|_{C^{1,\alpha}_{\mathrm{loc}}(\mathcal{C}_{(-L,L)}(R))} \leq C(L,R,R_0),
\]
for some $0<\alpha<1$.  
Hence, by the Ascoli–Arzelà theorem, we can extract a subsequence such that
\[
u_N \to u_0 \quad \text{in } C^{1,\alpha'}_{\mathrm{loc}}(\mathcal{C}_{(-L,L)}(R)),
\]
for any $0<\alpha'<\alpha$.  

Furthermore, using \eqref{eq:cn_bound1}, we deduce that
\begin{equation}\label{eq:3.101}
c_N(\cdot) \rightharpoonup^\ast c_0(\cdot) 
\quad \text{weakly* in } L^\infty(\mathcal{C}_{(-L,L)}(R)),
\end{equation}
up to subsequences.  

As a consequence, the limit $u_0$ satisfies
\[
\begin{cases}
-\Delta u_0 = c_0(x)\, u_0 & \text{in } \mathcal{C}_{(0,L)}(R), \\[0.3em]
u_0(x_1,x',x_n) \geq 0 & \text{in } \mathcal{C}_{(0,L)}(R), \\[0.3em]
u_0(x_1,x',0) = 0 & \text{on } \partial \mathcal{C}_{(0,L)}(R) \cap \partial \mathbb{R}^n_+ .
\end{cases}
\]

By the strong maximum principle, and recalling that $u_N(0,0',1)=1$ for every $N$, we infer that $u_0>0$ in $\mathcal{C}_{(0,L)}(R)$.  
Moreover, by \eqref{ipotesi_domini_piccoli}, the function $u_0$ satisfies $u_0\leq u_{0,\theta,h}$. By the Strong Comparison Principle and by the fact that $u_0>0$ in $\R^n_+$, we deduce that $u_0< u_{0,\theta,h}$, but this is an absurd, since $u_0(\tilde x_1,0,\tilde x_n)\geq u_{0,\theta,h}(\tilde x_1,0,\tilde x_n)$, by \eqref{contraddizione1}. This proves the claim \eqref{compattini}.

Since $w_{\theta',h'}\leq 0$ on $\partial (\mathcal{T}_{\theta',h'}\setminus K)$, we can apply Proposition \ref{domini_piccoli} to get that 
$$w_{\theta',h'}\leq 0 \quad \text{in }\mathcal{T}_{\theta',h'}\setminus K,$$
and therefore in the open set $\mathcal{T}_{\theta',h'}$. Also by the Strong Comparison Principle, since $w_{\theta ',h'}=0$ is not possible, we obtain 
   $$w_{\theta',h'}< 0 \quad \text{in }\mathcal{T}_{\theta',h'},$$
   and the proof is completed.
\end{proof}

Let us now show that, since it is possible to perform small translations and rotations 
of $\mathcal{T}_{\theta,h}$ towards $\mathcal{T}_{\theta',h'}$ whenever 
$(s',\theta')$ is close to $(s,\theta)$, one can in fact carry out larger translations 
and rotations as well. We state the following result.

\begin{lem}\label{grandispostamenti}
    Let $(\theta,h)$ and the set $\mathcal{T}_{\theta,h}$ be as above, and assume that
\begin{equation}\label{ipotesi_domini_piccoli2}
    w_{\theta,h} < 0 \quad \text{in } \mathcal{T}_{\theta,h}, \qquad 
w_{\theta,h} \le 0 \quad \text{on } \partial \mathcal{T}_{\theta,h}.
\end{equation}
\end{lem}
Let $(\hat \theta,\hat h)$ be fixed and assume that there exists a continuous function $g(t)=(\theta (t),h(t)):[0,1]\rightarrow (0,\pi/2)\times (0,+\infty),$ such that $g(0)=(\theta,h)$ and $g(1)=(\hat \theta,\hat h).$ Assume that $$w_{\theta(t),h(t)}\leq 0\quad \text{on } \partial (\mathcal T_{\theta(t),h(t)})\quad\text{for every }t\in [0,1).$$
Then we have
$$w_{\hat \theta,\hat h}<0\quad \text{in }\mathcal{T}_{\hat \theta,\hat h}.$$

\begin{proof}
    By Lemma \ref{piccoleperubations}, we get the existence of $\tilde t>0$ small such that 
    \begin{equation*}
        w_{\theta(t),h(t)}<0\quad \text{in }\mathcal{T}_{\theta(t),h(t)}.
    \end{equation*}
    We now set $$\bar T=\{\tilde t\in [0,1]\text{ such that } w_{\theta(t),h(t)}<0\text{ in } \mathcal{T}_{\theta(t),h(t)}\text{ for any }0\le t\le \tilde t\},$$
and $$\bar t=\sup \bar T.$$ We prove that $\bar t=1$. We assume that $\bar t<1$ and note that in this case, using the Strong Comparison Principle (by the Dirichlet condition), we have 
$$ w_{\theta(\bar t),h(\bar t)} < 0 \quad \text{in } \mathcal{T}_{\theta(\bar t),h(\bar t)}, \quad \text{and} \quad
w_{\theta(\bar t),h(\bar t)} \le 0 \quad \text{on } \partial \mathcal{T}_{\theta(\bar t),h(\bar t)}.$$
Using again Lemma \ref{piccoleperubations}, we can find a sufficiently small $\epsilon>0$ so that $$ w_{\theta(t),h(t)}<0\quad \text{in }\mathcal{T}_{\theta(t),h(t)},$$
for any $0\le t\le \tilde t +\epsilon$, which contradicts the definition of $\bar t$.

\end{proof}
\section{Monotonicity near the boundary}

\noindent This section is devoted to showing that any positive solution of \eqref{eq:principale} 
is increasing in the $x_n$-direction near the boundary $\partial \mathbb{R}^n_+$. 
We begin with the following preliminary result.

\begin{prop}\label{premoving}
    Let $u$ be a positive solution of \eqref{eq:principale}, with $f$ satisfying assumptions $(h_f)$. Assume that $u$ fulfills the condition \eqref{assunzionesullau}. Then there exist $\overline{h}>0$ and $\overline \theta \in (0,{\pi}/{2})$ such that 
    \begin{equation*}
        \frac{\partial u}{\partial V_\theta} >0 \quad \text{in }\Sigma_{{h},x_1} \quad \text{where}\quad \Sigma_{h,x_1}:=\{x\in \R^n: x_1=0,\quad 0\le x_n\le h\},
    \end{equation*}
    for every $(\theta,h)\in [-\overline{\theta},\overline{\theta}]\times [0,\bar h]$.
\end{prop}

\begin{proof}
    We argue by contradiction, and we assume that there exist $(\theta_N,h_N)\rightarrow (0,0)$ and $\overline{x}_N=(0,x'_N,x_{n,N})\in \Sigma_{h_N,x_1}$ such that
    \begin{equation}\label{contraddizione}
        \frac{\partial u}{\partial V_{\theta _N}} (\overline{x}_N) \leq 0  \qquad \text{and}\qquad x_{n,N} \to 0 \quad \text{as } N \to +\infty.
    \end{equation}
    We define the rescaled functions
\begin{equation}\label{eq:rescaled}
w_N(x_1,x',x_n) := \frac{u(x_1,x'+x'_N,x_n)}{u(0,x'_N,1)}.
\end{equation}
Clearly, $w_N(0,0',1) = 1$. Moreover, each $w_N$ satisfies
\begin{equation}\label{eq:wn_eq}
-\Delta w_N(x) = c_N(x) \, w_N(x),
\end{equation}
where
\begin{equation}\label{eq:cn_def}
c_N(x) := \frac{f(u(x_1,x'+x'_N,x_n))}{u(x_1,x'+x'_N,x_N)}.
\end{equation}

Proceeding in the same way as in the proof of Lemma~\ref{piccoleperubations}, exploiting the boundary Harnack,
we can extract a subsequence such that
\[
w_N \to w_0 \quad \text{in } C^{1,\alpha'}_{\mathrm{loc}}(\R^n),
\]
for some $0<\alpha'<1$.  

Moreover, we have that
\begin{equation}\label{eq:3.10}
c_N(\cdot) \rightharpoonup^\ast c_0(\cdot) 
\quad \text{weakly* in } L^\infty_{loc}(\mathcal{K}), 
\end{equation}
up to subsequences, and for any compact set $\mathcal{K}\subset \R^n$.  

As a consequence, the limit $w_0$ satisfies
\[
\begin{cases}
-\Delta w_0 = c_0(x)\, w_0 & \text{in } \R^n_+, \\[0.3em]
w_0(x_1,x',x_n) \ge 0 & \text{in } \R^n_+, \\[0.3em]
w_0(x_1,x',0) = 0 & \text{on } \partial  \mathbb{R}^N_+ .
\end{cases}
\]

By the strong maximum principle, and recalling that $w_N(0,1)=1$ for every $N$, we infer that $w_0>0$ in $\R^n_+$.  
Finally, by Hopf’s boundary lemma, it follows that
\[
\frac{\partial w_0}{\partial x_n}(0,0',0) > 0,
\]
which contradicts \eqref{contraddizione}, where instead we would have $
\frac{\partial w_0}{\partial x_n}(0,0',0) \leq 0$, since $V_{\theta_N}\rightarrow e_n$, and $x_{n,N}\rightarrow 0$.
\end{proof}

\begin{rem}\label{monotonia_nel_triangolino}
    Let $(\bar \theta,\bar h)$ be as above. Applying Proposition \ref{domini_piccoli}, we find the existence of $\bar h=\bar h(\bar \theta)$ small enough, such that \begin{equation*}\label{triangoli_piccoli2}
        w_{\bar \theta,\bar h}\leq 0 \quad\text{in }\mathcal T_{\bar \theta,\bar h}.
    \end{equation*}
    By the Dirichlet condition, using the  Strong Comparison Principle, we obtain 
    \begin{equation}\label{triangoli_piccoli}
        w_{\bar \theta,\bar h}< 0 \quad\text{in }\mathcal T_{\bar \theta,\bar h}.
    \end{equation}
    Moreover, by Proposition \ref{premoving}, we obtain 
    \begin{equation}\label{condizione_al_bordo}
        w_{\theta, h}\leq 0 \quad \text{on }\partial \mathcal{T}_{\theta, h},
    \end{equation}
    for any $(\theta,h)\in [-\bar \theta,\bar \theta]\times [0,\bar h].$
\end{rem}
\vspace{0.3cm}
Before proving the main result of this section, we recall the notation $$u_{\lambda}(x_1,x',x_n):=u(x_1,x',2\lambda-x_n).$$
\begin{prop}\label{monotonia_vicino_al_bordo}
Let $u$ be a positive weak solution of \eqref{eq:principale}, with $f$ satisfying assumptions $(h_f)$. We assume that $u$ fulfills the condition \eqref{assunzionesullau}. Then there exists $\hat \lambda>0$ such that, for any $0<\lambda\le \hat \lambda$, we have 
    $$u<u_\lambda\quad \text{in }\Sigma_\lambda.$$ Moreover, we deduce
    \begin{equation}\label{monotonia}
        \partial_{x_n}u>0\quad \text{in }\Sigma_\lambda.
    \end{equation}
\end{prop}

\begin{proof}
Let $(\bar \theta,\bar h)$ be as in \eqref{triangoli_piccoli}. We want to use Lemma \ref{grandispostamenti}. In this regard, for any fixed $(\theta',h')\in (0,\bar\theta)\times (0,\bar h)$, we consider the following function
\begin{equation*}
    g(t)=(\theta(t),h(t)):=(t\theta'+(1-t)\bar\theta,h'), \quad t\in [0,1].
\end{equation*}
We note that by \eqref{triangoli_piccoli} and by \eqref{condizione_al_bordo}, we can apply Lemma \ref{grandispostamenti}, and we obtain $w_{\theta',h'}<0$ in $\mathcal T_{\theta ',h'}.$ Therefore, since $0<\theta'<\bar\theta$ is arbitrary, by continuity we can pass to the 
limit as $\theta' \to 0$ and obtain
\[
u(x_1,x',x_n) \leq u_{h'}(x_1,x',x_n) \quad \text{in } \Sigma_{h'} \cap \{x_1 \geq 0\}, \qquad 0<h'<\bar{h}.
\]
By the invariance of the problem with respect to the axis $\{x_1=0\}$, the same 
argument applies for negative $\theta$, which yields
\[
u(x_1,x',x_n) \leq u_{h'}(x_1,x',x_n) \quad \text{in } \Sigma_{h'} \cap \{x_1 \leq 0\}, \qquad 0<h'<\bar{h},
\]
possibly after reducing $\bar{h}$. Hence we conclude that
\[
u(x_1,x',x_n) \leq u_{h'}(x_1,x',x_n) \quad \text{in } \Sigma_{h'}, \qquad \forall\, h' \in (0,\bar{h}).
\]

Now we set $\hat \lambda=\bar h$. Therefore, by Strong Comparison Principle and by Hopf’s Lemma, for every $\lambda \in (0,\hat\lambda]$ and every 
$(x_1,x') \in \mathbb{R}^{n-1}$, we obtain
\begin{equation*}
2\,\partial_{x_n} u(x_1,x',\lambda) 
= \partial_{x_n} \big(u-u_\lambda\big)(x_1,x',\lambda) > 0.
\end{equation*}
This proves \eqref{monotonia}.
\end{proof}

\section{Proof of the main results}
Now we introduce some notation. We set $$\Lambda:=\{\lambda>0 : u<u_{\lambda '}\text{ in } \Sigma_{\lambda'}\quad \forall \lambda'<\lambda \}.$$ By Proposition \ref{monotonia_vicino_al_bordo}, the set $\Lambda$ is not empty. Moreover, we define 
\begin{equation}\label{il_sup}
  \overline \lambda:=\sup \Lambda.   
\end{equation}
Our goal is to prove that $\overline{\lambda}=+\infty$. We note that, by the Strong Comparison Principle, if $\overline{\lambda}<+\infty$, and as above, we deduce 
\begin{equation}\label{cose1}
    u<u_\lambda\quad \text{and}\quad  
    \partial_{x_n}u>0\quad \text{in }\Sigma_\lambda,
\end{equation}
for any $\lambda\in (0,\overline{\lambda}].$
\vspace{0.2cm}

To establish our main results, we will rely on the following lemma.
\begin{lem}\label{lemmapreteoprincipale}
    Let $\overline\lambda$ be defined as in \eqref{il_sup}. Then there exists $\overline \delta>0$ such that for any $-\overline \delta\le\theta\le\overline\delta$ and for any $0<\lambda<\overline \lambda+\overline\delta$, we have 
    \begin{equation*}
        u<u_{\theta,\lambda} \quad \text{in }\Sigma_{\lambda,x_1}, \quad \text{where}\quad \Sigma_{\lambda,x_1}:=\{x\in \R^n: x_1=0,\quad 0\le x_n\le \lambda\}.
    \end{equation*}
\end{lem}

\begin{proof}
  We prove the result by contradiction. If we suppose that the result is false, then there exists a sequence of small $\delta_N\rightarrow 0$, together with parameters  $-\delta_N<\theta_N<\delta_N$, $0<\lambda_N<\overline \lambda+\delta_N$ and points $(0,x'_N,x_{n,N})\in \Sigma_{\lambda_N,x_1}$ such that 
  \begin{equation}\label{contraddizione3}
      u(0,x'_N,x_{n,N})\ge u_{\theta_N,\lambda_N}(0,x'_N,x_{n,N}),
  \end{equation}
with $0<x_{n,N}<\lambda_N$. Up to subsequences, we assume that $\lambda_N\rightarrow\tilde \lambda\le \overline{\lambda}$ and $x_{n,N}\rightarrow \tilde x_n\le \tilde \lambda.$ We observe that $\tilde{\lambda} > 0$; otherwise, inequality \eqref{contraddizione3} 
would contradict Proposition~\ref{premoving}. Now we consider the following sequences of functions given by 
$$w_N(x_1,x',x_n) := \frac{u(x_1,x'+x'_N,x_n)}{u(0,x'_N,1)}.$$
We note that, by \eqref{contraddizione3} and using the mean value theorem, we deduce that 
\begin{equation}\label{contraddizione4}
    \frac{\partial w_N}{\partial V_{\theta_N}}(\hat x_N)\leq 0,
\end{equation}
where $\hat x_N$ is a point lying on the line from $(0,0,x_{n,N})$ to $T_{\theta_N,\lambda_N}(0,0,x_{n,N})$.
As in the proof of Lemma \ref{piccoleperubations}, we obtain the existence of a function $u_0$ such that 

\[
u_N \to u_0 \quad \text{in } C^{1,\alpha'}_{\mathrm{loc}}(\R^n),
\]
for some $0<\alpha '<1$.
By means of a diagonal argument, we can construct a function $u_0$ such that, in the limit,

\[
\begin{cases}
-\Delta u_0 = c_0(x)\, u_0 & \text{in } \R^n_+, \\[0.3em]
u_0(x_1,x',x_n) \ge 0 & \text{in } \R^n_+, \\[0.3em]
u_0(x_1,x',0) = 0 & \text{on } \partial \R^n_+ .
\end{cases}
\]
Since $u < u_{\tilde\lambda}$ in $\Sigma_{\tilde\lambda}$ (see \eqref{cose1}), by the 
definition of $w_N$ we obtain that 
\[
u_0 \leq u_{0,\tilde\lambda} \quad \text{in } \Sigma_{\tilde\lambda}.
\] 
Moreover, since $u_0(0,0',1)=1$, the Strong Comparison Principle (using the Dirichlet 
boundary condition) yields
\[
u_0 < u_{0,\tilde\lambda} \quad \text{in } \Sigma_{\tilde\lambda}.
\]

Now, if $x_{n,N} \to \tilde{x}_n < \tilde{\lambda}$, then by \eqref{contraddizione3} 
we arrive at a contradiction. On the other hand, passing to the limit in 
\eqref{contraddizione4} gives
\[
\frac{\partial u_0}{\partial x_n}(0,0',\tilde{\lambda}) \leq 0,
\]
which contradicts the fact that 
\[
\frac{\partial u_0}{\partial x_n}(0,0',\tilde{\lambda}) > 0,
\]
by Hopf's Lemma.
  
\end{proof}

\begin{proof}[Proof of Theorem \ref{teo:Monotonia}]
We want to prove that $\overline \lambda=+\infty$. We argue by contradiction, and we assume that $\overline \lambda<+\infty.$ By Remark \ref{monotonia_nel_triangolino}, we obtain the existence of $(\bar\theta,\bar h)$ such that $w_{\bar\theta,\bar h}<0$ in $\mathcal{T}_{\bar\theta,\bar h}.$ Then, recalling the value $\overline\delta$ given in Lemma \ref{lemmapreteoprincipale},  we fix $\theta_0>0$  with $\theta_0\leq \overline \delta$ and $\theta_0\leq \bar\theta$. Let us set $$h_0:=h_0(\theta_0),$$ 
such that the set $\mathcal{T}_{\theta_0,h_0}$ is contained in $\mathcal{T}_{\bar\theta,\bar h}$. In this way, see also Proposition \ref{domini_piccoli}, $w_{\theta_0,h_0}<0$ in $\mathcal{T}_{\theta_0,h_0}$. It is convenient to assume that $h_0 \leq \hat{\lambda}$, with $\hat{\lambda}$ as in Proposition \ref{monotonia_vicino_al_bordo}. 
For any $h_0 < h \leq \overline{\lambda}+\bar{\delta}$ and $0<\theta<\theta_0$, we apply the 
sliding-rotating method using Lemma \ref{grandispostamenti} with
\[
g(t) = (\theta(t),h(t)) := \big( t\theta+(1-t)\theta_0,\, th+(1-t)h_0\big), 
\qquad t \in [0,1].
\]
By Lemma \ref{lemmapreteoprincipale}, the boundary conditions required to apply Lemma \ref{grandispostamenti} are satisfied, and 
hence, by Lemma \ref{grandispostamenti}, we conclude that $w_{\theta,h}<0$ in $\mathcal{T}_{\theta,h}$. 

Proceeding as in the proof of Proposition \ref{monotonia_vicino_al_bordo}, we deduce that 
\[
u(x_1,x',x_n) < u_\lambda(x_1,x',x_n) \quad \text{in } \Sigma_\lambda \quad \text{for all } 0<\lambda \leq \bar{\lambda}+\bar{\delta}.
\]
This leads to a contradiction unless $\bar{\lambda}=+\infty$. Finally, arguing as in 
the proof of Proposition \ref{monotonia_vicino_al_bordo}, we obtain
\[
\partial_{x_n} u > 0 \quad \text{in } \mathbb{R}^n_+ 
.
\]

\end{proof}

We are ready to prove Theorem \ref{teo:Lane-Emden}.

\begin{proof}[Proof of Theorem  \ref{teo:Lane-Emden}]
	First of all we observe that any nonnegative solution is actually positive if it is not trivial. 
As a consequence of our Theorem \ref{teo:Monotonia}, any positive solution of \eqref{eq:laneemden} is monotone increasing in the $x_n$-direction. The claim then follows directly from \cite[Theorem 1.1]{DSS}, or from \cite{DFT} since any monotone solution is also stable. 
\end{proof}

\section*{Acknowledgements}
B. Sciuzni and D. Vuono have been supported by PRIN PNRR P2022YFAJH \emph{Linear and Nonlinear PDEs: New directions and applications.}  The authors have been partially supported by \emph{INdAM-GNAMPA Project Regularity and qualitative aspects of nonlinear PDEs via variational and non-variational approaches} E5324001950001.

\vspace{0.2 cm}
  
\textbf{Declarations.} Data sharing is not applicable to this article as no datasets were generated or analyzed during the current study.
Conflicts of interest: The authors have no conflicts of interest to declare.

\end{document}